\documentclass[reqno,11pt]{amsart}
\usepackage{amsmath,amssymb,latexsym,soul,cite,mathrsfs}
\usepackage{color,enumitem,graphicx}
\usepackage[colorlinks=true,urlcolor=blue,
citecolor=red,linkcolor=blue,linktocpage,pdfpagelabels,
bookmarksnumbered,bookmarksopen]{hyperref}
\usepackage[english]{babel}
\usepackage[left=2.8cm,right=2.8cm,top=2.8cm,bottom=2.8cm]{geometry}
\usepackage[hyperpageref]{backref}
\usepackage{lipsum}

 \numberwithin{equation}{section}

 \usepackage{mathtools, nccmath}

\usepackage{bbm}

\usepackage{tcolorbox}

\newtheorem{theorem}{Theorem}[section]
\newtheorem{lemma}[theorem]{Lemma}
\newtheorem{corollary}[theorem]{Corollary}

\theoremstyle{definition}

\theoremstyle{remark}

\newcommand{\T}{\mathbb T}
\newcommand{\R}{\mathbb R}
\newcommand{\cD}{\Gamma}
\newcommand{\cR}{\mathcal R}
\newcommand{\ep}{\varepsilon}

\usepackage{etoolbox}
\makeatletter
\patchcmd{\@maketitle}
  {\ifx\@empty\@dedicatory}
  {\ifx\@empty\@date \else {\vskip3ex \centering\footnotesize\@date\par\vskip1ex}\fi
   \ifx\@empty\@dedicatory}
  {}{}
\patchcmd{\@adminfootnotes}
  {\ifx\@empty\@date\else \@footnotetext{\@setdate}\fi}
  {}{}{}
\makeatother

\title{Global regularity of critical SQG via a pointwise Leibniz commutator formula}
\author[M. Caselli]{Michele Caselli}

\address[M. Caselli]{Princeton University --  Fine Hall, 304 Washington Rd, Princeton, NJ 08540, USA
	\newline\indent 
	University of Sydney -- Quadrangle A14, Camperdown, NSW 2006, Australia}
\email{mc3147@princeton.edu}

\date{\today}

\begin{document}

\begin{abstract}
We give a short proof of global regularity for the critical dissipative SQG equation in two dimensions. The argument is based on a pointwise representation of the Leibniz commutator for the half-Laplacian, recently obtained by the author with Luca Gennaioli, which simplifies and unifies the nonlinear estimates used in previous maximum principle proofs.
\end{abstract}

\maketitle

\section{Introduction}

In this work, we give a short proof of global regularity for the critical dissipative surface quasi-geostrophic (SQG) equation
\begin{equation}\label{eq:sqg-intro}
    \begin{cases}
        \partial_t\theta+u\cdot\nabla\theta+\Lambda\theta=0, \\ u=\cR^\perp\theta=(-R_2\theta,R_1\theta) , \\ \theta(\cdot,0)=\theta_0 , 
    \end{cases}
\end{equation}
on the two-dimensional periodic domain $\T^2=[-\pi, \pi]^2$. Here $\Lambda = (-\Delta)^{1/2}$ and $R_1, R_2$ are the Riesz transforms, with the convention \(R_j=\partial_j\Lambda^{-1}\) on nonzero Fourier modes and \(R_j=0\) on the zero mode, so that \(R_j\Lambda=\partial_j\).

The surface quasi-geostrophic (SQG) equation was introduced in \cite{CMT94} and, to date, its regularity theory has been developed through several methods. Global regularity for arbitrary smooth initial data was proved independently in \cite{KNV07} by a very elegant argument involving a preserved modulus of continuity, and in \cite{CV10} through a De Giorgi type iteration. The result in \cite{CV10} was also re-proved using elementary techniques \cite{KN09}. Later, global regularity was also recovered by nonlinear methods. Our proof is closer to the nonlinear maximum principle approach \cite{CV12} and follows, more specifically, the finite difference strategy of \cite{CTV15}, where one first propagates a small H\"older exponent by studying finite differences and then uses this estimate to control the gradient.

The novelty of our argument is that the nonlinear estimates required for the maximum principle proof are derived from the same pointwise Leibniz formula recently proved by the author with Luca Gennaioli in \cite{CG26}, which is, in turn, based on the Caffarelli--Silvestre extension \cite{CS07}. This observation replaces the separate ad hoc singular kernel decompositions used in earlier proofs and leads to a particularly short proof of global regularity. Precisely, we give a new proof of the following classical theorem.

\begin{theorem}\label{thm:main}
For every \(\theta_0\in C^\infty(\T^2)\), \eqref{eq:sqg-intro}
has a unique solution $  \theta\in C^\infty(\T^2\times[0,\infty))$, and
\[
 \sup_{t\ge0}\|\nabla\theta(t)\|_{L^\infty} \le C\big(\|\theta_0\|_{L^\infty} , \|\nabla \theta_0\|_{L^\infty} \big) . 
\]
\end{theorem}

\subsection{Main idea and proof outline}

 Let $P_z=e^{-z\Lambda}$ be the Poisson semigroup on $\T^2$ and let $F(x,z)=P_zf(x)$ be the Caffarelli--Silvestre extension of $f$, which is the harmonic extension to $\T^2 \times [0, +\infty)$ in this case. Here we identify the periodic domain $\T^2$ with the closed Riemannian manifold endowed with its standard flat metric. Under this identification, the fractional Laplacian obtained by periodizing the fractional Laplacian on $\R^2$ coincides with the spectral fractional Laplacian on $\T^2$; see \cite{RoncalStinga16}. Then, by \cite[Theorem 1.2]{CG26} the Leibniz commutator 
\begin{equation}\label{eq: Gamma def}
 \cD[f]:=f\cdot\Lambda f-\frac12\Lambda|f|^2
\end{equation}
can be written as
\begin{equation}\label{eq: CG Leibniz}
 \cD[f](x)=\int_0^\infty
 P_z\bigl(|\widetilde\nabla F(\cdot, z)|^2\bigr)(x) \, dz ,
\end{equation}
where $\widetilde\nabla=(\nabla_x,\partial_z)$. The commutator $\Gamma[f]$ and its nonnegative singular integral representation are very classical; the new ingredient from \cite{CG26} is the exact extension representation \eqref{eq: CG Leibniz}. Expressing the Leibniz commutator in this way easily yields two key consequences: for every $R,z>0$
\begin{equation}\label{eq: point for consequences}
    \cD[f](x)\ge \frac{|f(x)-P_{2R}f(x)|^2}{4R}, \qquad \mbox{and} \qquad |\cR(f-P_zf)(x)|\lesssim \sqrt{z\cD[f](x)} .  
\end{equation}
The first estimate provides the (pointwise) nonlinear coercivity of the dissipation, while the second one controls the high-frequency part of the Riesz transform. With these estimates, we recover both the finite difference and gradient estimates needed for the argument (see Corollary \ref{cor:nonlinear}). 

In \cite{CTV15}, see also \cite{CZV16}, these types of estimates are obtained from the singular-integral representations of $\Lambda$ and the Riesz transforms by splitting the kernels at a suitable scale. Here these kernel computations are replaced by the exact extension identity above. Thus both the coercivity of the dissipation and the control of the drift follow from \eqref{eq: CG Leibniz}.

From here, the proof of Theorem \ref{thm:main} is then quite straightforward. For $h\in\R^2$, we consider the normalized finite differences
\begin{equation*}
 v(x,h,t):=\frac{\theta(x+h,t)-\theta(x,t)}{|h|^\alpha}, 
\end{equation*}
so that a uniform bound on $\sup_{x,h}|v|$ is equivalent to a bound on $[\theta]_{C^\alpha}$. Computing the equation for $v^2$ we see that it contains a ``bad" term involving the velocity increment $\tau_hu$. Combining the two consequences of the extension formula \eqref{eq: point for consequences}, and choosing $\alpha>0$ sufficiently small in terms of $\|\theta_0\|_{L^\infty}$, we absorb this term into the one involving $\Gamma[\tau_h\theta]$ which has a sign. We obtain a closed maximum principle for $v^2$, which propagates the $C^\alpha$ seminorm of the initial datum.

Lastly, we differentiate \eqref{eq:sqg-intro}. The same extension estimates, now applied to $\nabla\theta$, show that the nonlinear dissipation dominates the stretching term whenever $|\nabla\theta|$ is large. Then, a suitable maximum principle argument yields a uniform Lipschitz bound, and a standard continuation criterion in time concludes.

\section{Consequences of the Leibniz commutator formula}\label{sec:defect}

The next lemma is the main consequence of the extension formula \eqref{eq: CG Leibniz}.

\begin{lemma}\label{lem:tube}
Let $f \in C^\infty(\T^2; \R^k )$, $F(x,z)=P_z f(x)$ be its (componentwise) extension, and $\Gamma[f]$ be the Leibniz commutator \eqref{eq: Gamma def} summed over the components of $f$. Then, for every $R,z>0$ and every $x\in\T^2$,
\begin{align}
 \cD[f](x)
 &\ge \frac{|f(x)-P_{2R}f(x)|^2}{4R},
 \label{eq:tube-coercivity}\\
 \big|\cR(f-P_zf)(x)\big|
 &\lesssim \sqrt{z\cD[f](x)}.
 \label{eq:riesz-high}
\end{align}

\end{lemma}

\begin{proof}
With a little abuse of notation, we drop the dependence on the point $x \in \T^2$, which is fixed throughout the proof. By the semigroup property 
\begin{equation*}
 \widetilde\nabla F(2s)=P_s\widetilde\nabla F(s).
\end{equation*}
Indeed, this is immediate for spatial derivatives, while for the vertical derivative
\begin{equation*}
 (\partial_z F)(2s)=-\Lambda F(2s)=-P_s\Lambda F(s)
 =P_s(\partial_z F)(s).
\end{equation*}
Hence, Jensen's inequality gives
\begin{equation*}
 |\widetilde\nabla F(2s)|^2 = |P_s \widetilde\nabla F(s)|^2
 \le P_s\bigl(|\widetilde\nabla F(s)|^2\bigr).
\end{equation*}
Consequently, \eqref{eq: CG Leibniz} and Cauchy--Schwarz in the vertical
variable imply
\begin{align*}
    \cD[f]  =\int_0^\infty
 P_z \big(|\widetilde\nabla F(z)|^2\big) \, dz & \ge \int_0^\infty |\widetilde\nabla F(2z)|^2 \, dz \\ & \ge \int_0^R  |\partial_z F(2z)|^2 \, dz = \frac{1}{2} \int_0^{2R} |\partial_s F(s)|^2 \, ds  \ge \frac{|f-F(2R)|^2}{4R},
 \end{align*}
which proves \eqref{eq:tube-coercivity}.

For the Riesz part, using that \(R_j\Lambda=\partial_j\) we can write
\[
 R_j(f-F(z))=  - R_j \int_0^z \partial_s F(s) \, ds = \int_0^z R_j \Lambda F(s) \, ds = \int_0^z \partial_j F(s) \, ds . 
\]
Then, the same semigroup and Jensen estimate described above, now applied to horizontal derivatives, yields
\begin{equation*}
    | \mathcal{R} (f-F(z))|^2 \le z \int_0^z |\nabla F(s)|^2 \, ds \le 2 z \int_0^{z/2} P_s \big( |\nabla F(s)|^2 \big) \, ds \le 2 z \Gamma[f], 
\end{equation*}
which gives \eqref{eq:riesz-high}. 
\end{proof}

For finite differences, we identify functions on $\mathbb T^2$ with their
$2\pi$-periodic extensions to $\mathbb R^2$ and, for $h\in\mathbb R^2$, we write
\[
\tau_h \theta(x)=\theta(x+h)-\theta(x).
\]

The estimates collected in the following corollary are already contained, in essentially equivalent form, in the nonlinear lower bounds in \cite[Eq.~(3.41) and Remark~3.5]{CV12}, \cite[Eq.~(4.16), (4.25), and (5.16)]{CTV15}, and \cite[Lemmas~3.1 and~3.3]{CZV16}. Here we give a short unified derivation from our pointwise identity \eqref{eq: CG Leibniz}. 

\begin{corollary}\label{cor:nonlinear}
Let $\theta \in C^\infty(\T^2)$ be non-constant. Let $u=\cR^\perp\theta$ and $\Gamma $ be the Leibniz commutator \eqref{eq: Gamma def}. Then, for $h\in\R^2\setminus\{0\}$, 
\begin{align}
 \cD[\tau_h \theta](x)& \gtrsim \frac{|\tau_h\theta(x)|^3}{ \|\theta\|_{L^\infty}|h|},
 \label{eq:difference-lower}\\
 |\tau_h u(x)|&\lesssim  
       \sqrt{z\cD[\tau_h\theta](x)}+\frac{ \|\theta\|_{L^\infty} |h|}{z} ,\qquad \forall \, z>0.
 \label{eq:difference-velocity}
\end{align}
Moreover, if $0<\alpha<1$ then
\begin{align}
 \cD[\nabla \theta](x)& \gtrsim_\alpha \frac{|\nabla \theta(x)|^{2+\frac{1}{1-\alpha}}}
 {[\theta]_{C^\alpha}^{\frac{1}{1-\alpha}}},
 \label{eq:gradient-lower}\\
 |\nabla u(x)|& \lesssim_\alpha 
       \sqrt{z \cD[\nabla \theta](x)}+[\theta]_{C^\alpha}z^{\alpha-1}
   ,\qquad \forall \, z>0.
 \label{eq:strain}
\end{align}
\end{corollary}

\begin{proof}
By \eqref{eq:tube-coercivity} applied to $f=\tau_h \theta$ we have 
\begin{equation*}
 \cD[\tau_h\theta](x)
 \ge \frac{|\tau_h\theta(x)-P_{2R}(\tau_h\theta)(x)|^2}{4R}.
\end{equation*}
Since $P_z(\tau_h\theta)=\tau_hP_z\theta$, the estimate in \eqref{eq:poisson-Linf} gives
\begin{equation*}
 |P_{2R}(\tau_h\theta)(x)|
 \lesssim  |h|\|\nabla P_{2R}\theta\|_{L^\infty}
 \lesssim \frac{ \|\theta\|_{L^\infty} |h|}{R}.
\end{equation*}
Let $A=|\tau_h\theta(x)|$. If $A=0$, there is nothing to prove. Otherwise, choose $R=C\|\theta\|_{L^\infty}|h|/A$, where $C>0$ is a sufficiently large universal constant. Then $|P_{2R}(\tau_h\theta)(x)|\le A/2$, and hence
\begin{equation*}
 \cD[\tau_h\theta](x)\gtrsim \frac{A^2}{R}
 = \frac{A^3}{C \|\theta\|_{L^\infty}|h|},
\end{equation*}
which is \eqref{eq:difference-lower}. 

Next, we split
\begin{equation*}
 \tau_hu
 =\cR^\perp(\tau_h\theta-P_z\tau_h\theta)
 +\tau_h\cR^\perp P_z\theta.
\end{equation*}
The first term can be bounded by \eqref{eq:riesz-high}, while the second term is at most
\(C|h|\|\nabla\cR P_z\theta\|_{L^\infty}\), so the estimate in
\eqref{eq:poisson-Linf} proves \eqref{eq:difference-velocity}.

For the gradient estimate we argue similarly. If $g:=|\nabla \theta (x)|=0$ the lower bound is trivial. Otherwise, we have $[\theta]_{C^\alpha}>0$ and by \eqref{eq:poisson-holder} 
\begin{equation*}
 |P_{2R} \nabla \theta (x)|\le C_\alpha [\theta]_{C^\alpha} R^{\alpha-1}.
\end{equation*}
Choose $R=(C_\alpha [\theta]_{C^\alpha}/g)^{1/(1-\alpha)}$, where $C_\alpha>0$ is sufficiently large so that $|P_{2R} \nabla \theta (x)|\le g/2$. Then \eqref{eq:tube-coercivity} proves \eqref{eq:gradient-lower}. Finally,
\begin{equation*}
 \nabla u=\cR^\perp \nabla \theta
 =\cR^\perp(\nabla \theta -P_z \nabla \theta)+\nabla\cR^\perp P_z \theta;
\end{equation*}
use \eqref{eq:riesz-high} and \eqref{eq:poisson-holder} to obtain
\eqref{eq:strain}.
\end{proof}

\section{Proof of Theorem \ref{thm:main}}

\subsection{The finite difference maximum principle}\label{sec:finite-differences}

We now use the two estimates from Corollary~\ref{cor:nonlinear}
to obtain a uniform H\"older bound in time for $\theta$. All estimates below are first made on the maximal interval of existence of a
smooth solution $[0, T_{\max})$, and the bounds are independent of the endpoint $T_{\max}$. We study the finite difference
\begin{equation*}
 v(x,h,t)=\frac{\tau_h\theta(x,t)}{|h|^\alpha},
 \qquad x\in\T^2,\quad h\in\R^2\setminus\{0\}.
\end{equation*}

Let $\alpha \in (0,1)$ to be chosen later. By the maximum principle (see \cite[Theorem 4.1]{CC04})
\begin{equation}\label{eq:Linf}
\sup_{t \in [0, T_{\max})} \|\theta(t)\|_{L^\infty}
= M:=\|\theta_0\|_{L^\infty}.
\end{equation}
If $\theta_0$ is constant there is nothing to prove, hence we assume $M>0$ and $[\theta_0]_{C^\alpha} > 0 $ throughout the proof. Subtracting the equation at $x+h$ and $x$ yields
\begin{equation*}
L_h (\tau_h\theta  ) :=  \left(
\partial_t
+u(x)\cdot\nabla_x
+\tau_h u(x)\cdot\nabla_h
+\Lambda_x
\right)\tau_h\theta=0.
\end{equation*}
Since the weight $|h|^{-2\alpha}$ is independent of $x$ and $t$, we have 
\begin{align*}
    L_h v^2 & = L_h \left( \frac{|\tau_h\theta(x,t)|^2}{|h|^{2\alpha}} \right) \\ &= \frac{1}{|h|^{2\alpha}} \Big( \partial_t |\tau_h\theta|^2 + u \cdot \nabla_x |\tau_h\theta|^2 + \Lambda_x |\tau_h\theta|^2 \Big) + \tau_h u \cdot \left( \frac{\nabla_h |\tau_h\theta|^2}{|h|^{2\alpha}} + \nabla_h(|h|^{-2\alpha})|\tau_h\theta|^2 \right) \\ &= \frac{2 \tau_h \theta}{|h|^{2\alpha}} \Big( \partial_t \tau_h\theta + u \cdot \nabla_x \tau_h\theta + \tau_h u \cdot \nabla_h \tau_h \theta \Big) + \frac{1}{|h|^{2\alpha}}\Lambda_x |\tau_h\theta|^2 + \tau_h u \cdot \nabla_h(|h|^{-2\alpha})|\tau_h\theta|^2 \\&=  \frac{2 \tau_h \theta}{|h|^{2\alpha}} L_h (\tau_h \theta) - \frac{2}{|h|^{2\alpha}} \Gamma[\tau_h \theta] + \tau_h u \cdot \nabla_h(|h|^{-2\alpha})|\tau_h\theta|^2 \\ &=  - \frac{2}{|h|^{2\alpha}} \Gamma[\tau_h \theta] + \tau_h u \cdot \nabla_h(|h|^{-2\alpha})|\tau_h\theta|^2 , 
\end{align*}
where we have used that $L_h(\tau_h \theta)=0$ in the last line.

By the elementary fact that, for $A,B \ge 0$, 
\begin{equation*}
    \inf_{z > 0} \left(\sqrt{Az} + \frac{B}{z}\right)  = 3 \left( \frac{AB}{4}\right)^{\frac{1}{3}} , 
\end{equation*}
optimizing the right-hand side of \eqref{eq:difference-velocity} gives
\begin{equation*}
    |\tau_h u| \lesssim \big(\Gamma[\tau_h \theta] M |h|\big)^{\frac{1}{3}} . 
\end{equation*}
Moreover, by \eqref{eq:difference-lower} we also have
\begin{equation*}
    |\tau_h \theta|^2  \lesssim \big(\Gamma[\tau_h \theta] M |h|\big)^{\frac{2}{3}} . 
\end{equation*}
Hence, we can estimate 
\begin{align*}
    |\tau_h u \cdot \nabla_h(|h|^{-2\alpha}) | |\tau_h\theta|^2 \lesssim \frac{\alpha}{|h|^{2\alpha+1}}  \cdot \big(\Gamma[\tau_h \theta] M |h|\big)^{\frac{1}{3}+\frac{2}{3}}  =  C_0 \alpha M \frac{\Gamma[\tau_h \theta]}{|h|^{2\alpha}} . 
\end{align*}
Choosing $\alpha = \min\left\{ \frac{1}{10}, \frac{1}{10C_0M} \right\}$ gives 
\begin{align*}
    L_h v^2 \le  - \frac{2}{|h|^{2\alpha}} \Gamma[\tau_h \theta] + \frac{1}{|h|^{2\alpha}} \Gamma[\tau_h \theta] = - \frac{1}{|h|^{2\alpha}} \Gamma[\tau_h \theta] \le 0 . 
\end{align*}

For every \(T<T_{\max}\) fixed, we apply Lemma~\ref{lem:joint-maximum} on \([0,T]\) with
\[
w=v^2,\qquad b=u,\qquad c=\tau_h u.
\]
Indeed, smoothness gives 
\[
 \sup_{x\in \T^2, t\in [0,T]} v^2(x,h,t)=O(|h|^{2-2\alpha})\to 0
 \qquad\text{as }h\to0, 
\]
Thus $v^2$ extends continuously to $h=0$ by setting $v^2(x,0,t)=0$.
Moreover, \eqref{eq:Linf} gives
\[
  \sup_{x\in \T^2, t\in [0,T]} v^2(x,h,t)
 \le \frac{4M^2}{|h|^{2\alpha}}
 \to0
 \qquad\text{as }|h|\to\infty,
\]
Hence Lemma~\ref{lem:joint-maximum} and letting $T \uparrow T_{\max}$ yield
\[
\sup_{x\in\mathbb T^2, h\in\mathbb R^2}
v^2(x,h,t)
\le
\sup_{x\in\mathbb T^2, h\in\mathbb R^2}
v^2(x,h,0) ,
\qquad \forall \,  t \in [0,T_{\max}).
\]

Since $\theta$ is $2\pi$-periodic, the quantity
\[
\sup_{x\in\mathbb T^2,  h \neq 0}
\frac{|\tau_h\theta(x)|}{|h|^\alpha}
\]
is the usual $C^\alpha(\mathbb T^2)$ seminorm. Indeed, for any $h\in\R^2$, one may replace $h$ by a shortest representative $h-2\pi k$, $k\in\mathbb Z^2$, without changing the increment and only decreasing the denominator. Conversely, every pair of points in $\T^2$ admits such a shortest representative, so the two suprema agree.
Hence 
\begin{equation*}
\sup_{t \in [0, T_{\max})} [\theta(t)]_{C^\alpha}
\le
K
:= [\theta_0]_{C^\alpha} <\infty . 
\end{equation*}

\subsection{From H\"older control to a gradient bound}\label{sec:gradient}

 Differentiating
the SQG equation, taking the scalar product with \(\nabla \theta \), and using
the definition of $\Gamma$ componentwise, gives
\begin{equation}\label{eq:grad theta equation}
\frac12(\partial_t+u\cdot\nabla+\Lambda)|\nabla \theta|^2 
 =- \partial_i u_j \partial_i\theta \partial_j \theta - \cD[ \nabla \theta] 
 \le |\nabla u| |\nabla \theta|^2 - \cD[ \nabla \theta] . 
\end{equation}
By the elementary fact that, for $A,B \ge 0$ and $\alpha \in (0,1)$, 
\begin{equation*}
    \inf_{z > 0} \left(\sqrt{Az} + \frac{B}{z^{1-\alpha}}\right)  = C_\alpha A^{\frac{1-\alpha}{3-2\alpha}} B^{\frac{1}{3-2\alpha}} , 
\end{equation*}
optimizing the right-hand side of \eqref{eq:strain} gives
\begin{equation*}
    |\nabla u| \lesssim_\alpha \Gamma[\nabla \theta]^{\frac{1-\alpha}{3-2\alpha}} K^{\frac{1}{3-2\alpha}}. 
\end{equation*}
Moreover, by \eqref{eq:gradient-lower} we also have
\begin{equation*}
    |\nabla \theta|^2  \lesssim_\alpha \Gamma[\nabla \theta]^{\frac{2(1-\alpha)}{3-2\alpha}} K^{\frac{2}{3-2\alpha}} .
\end{equation*}
Hence, by the last two estimates and Young's inequality with exponents 
\begin{equation*}
    p= \frac{3-2\alpha}{3(1-\alpha)} , \qquad q=\frac{3-2\alpha}{\alpha} , 
\end{equation*}
we have 
\begin{align*}
    |\nabla u||\nabla \theta|^2 \lesssim_\alpha \Gamma[\nabla \theta]^{\frac{3(1-\alpha)}{3-2\alpha}} K^{\frac{3}{3-2\alpha}} \le \frac{1}{2} \Gamma[\nabla \theta] + C_\alpha K^{\frac{3}{\alpha}} . 
\end{align*}
Using this estimate in \eqref{eq:grad theta equation} gives
\begin{equation*}
\frac12(\partial_t+u\cdot\nabla+\Lambda)|\nabla \theta|^2 \le -    \frac{1}{2} \Gamma[\nabla \theta] + C_\alpha K^{\frac{3}{\alpha}} . 
\end{equation*}

We claim that the right-hand side is $\le 0$ whenever 
\begin{equation*}
    |\nabla \theta| \ge C K^{\frac{1}{\alpha}} , 
\end{equation*}
where $C>0$ is a sufficiently large constant that depends only on $\alpha$. Indeed, if this is the case, by \eqref{eq:gradient-lower} again
\begin{equation*}
    \Gamma[\nabla \theta] \gtrsim_\alpha  \frac{|\nabla \theta|^{\frac{3-2\alpha}{1-\alpha}}}{K^{\frac{1}{1-\alpha}}} \gtrsim_\alpha  C^{\frac{3-2\alpha}{1-\alpha}} K^{\frac{3}{\alpha}} \ge 2C_\alpha K^{\frac{3}{\alpha}} , 
\end{equation*}
for $C$ sufficiently large.

Thus there is a constant \(C_\alpha>0\) such that
\begin{equation}\label{eq:w-subsolution}
 (\partial_t+u\cdot\nabla+\Lambda) |\nabla \theta|^2 \le 0
 \qquad\hbox{wherever } |\nabla \theta|\ge C_\alpha K^{\frac{1}{\alpha}}.
\end{equation}
Now we can conclude by an application of the maximum principle. Define $w:=|\nabla \theta|^2$, fix \(A > \max \left\{\|\nabla\theta_0\|_{L^\infty}^2, C_\alpha^2 K^{2/\alpha} \right\} \) and set $\Phi(r)=(r-A)_+^3$. This is a convex function with $\Phi \ge 0$, $\Phi(0)=0$, and $\Phi(w(\cdot,0))=0$ since $A>w(\cdot,0)$. Since $\Phi'$ is supported in the region where \eqref{eq:w-subsolution} holds, by \eqref{eq:w-subsolution} and the Córdoba-Córdoba inequality (see \cite{CordobaCordoba2003, CC04} or \cite[Theorem 1.5]{CG26})
\[
 (\partial_t+u\cdot\nabla+\Lambda)\Phi(w)
 =\Phi'(w)(\partial_t+u\cdot\nabla+\Lambda)w
  -\big(\Phi'(w)\Lambda w-\Lambda\Phi(w)\big)
 \le0.
\]
Then, the maximum principle implies \(\Phi(w)\equiv0\). Letting
\[
A\downarrow \max \left\{\|\nabla\theta_0\|_{L^\infty}^2, C_\alpha^2 K^{2/\alpha} \right\}
\]
gives
\begin{equation*}
 \sup_{ t \in [0,T_{\max})} \|\nabla\theta(t)\|_{L^\infty}
 \le\max\left\{\|\nabla\theta_0\|_{L^\infty},
                 C_\alpha K^{1/\alpha}\right\}<\infty.
\end{equation*}
Moreover, since $K \le {\rm diam} (\T^2)^{1-\alpha} \|\nabla \theta_0\|_{L^\infty}$ and $\alpha$ has been chosen depending only on $M=\|\theta_0\|_{L^\infty}$, the uniform bound for $\|\nabla\theta(t)\|_{L^\infty}$ depends only on $\| \theta_0\|_{L^\infty}$ and $\|\nabla \theta_0\|_{L^\infty}$. 

\subsection{Conclusion} Let \([0,T_{\max})\) be the maximal interval of existence of the
classical solution.  The preceding argument gives the uniform estimate
\[
  \sup_{t \in [0,T_{\max})}
  \|\nabla\theta(t)\|_{L^\infty}
 \le C\big( \|\theta_0\|_{L^\infty},\| \nabla \theta_0\|_{L^\infty}\big) . 
\]
The continuation criterion for periodic critical SQG (see, for example, the few lines after \cite[Proposition~4.2]{CTV15}) therefore allows the solution to be
continued past any finite $T_{\max}$. Hence $T_{\max}=+\infty$, and
Theorem~\ref{thm:main} follows.

\appendix

\section{}
\label{app:joint-maximum}

The following argument is classical and, for example, essentially contained at the end of the proof of Theorem~4.3 in \cite{CTV15}. In that work the periodicity of
the increment is used to reduce the displacement variable to a compact
set. We record here the corresponding statement on $\R^2$, with two limiting
conditions to prevent a maximum from escaping through $h=0$ or
$|h|=\infty$.

\begin{lemma}\label{lem:joint-maximum}
Let $T>0$ and 
\[
 b,c:\T^2\times(\R^2\setminus\{0\})\times[0,T]\to\R^2
\]
be continuous. Suppose that
\[
 w:\T^2\times\R^2\times[0,T]\to[0,\infty)
\]
is continuous, and that on
$\mathcal{U} := \T^2\times(\R^2\setminus\{0\})\times(0,T]$ it is smooth. Assume
\begin{equation}\label{eq:joint-subsolution}
\bigl(\partial_t+b\cdot\nabla_x+c\cdot\nabla_h+\Lambda_x\bigr)w
 \le 0 \qquad \mbox{in } \, \mathcal{U} , 
\end{equation}
and that
\begin{equation}
 \lim_{|h|\to 0}
 \sup_{\substack{x\in\T^2, t\in[0,T] }}w(x,h,t) =0,    \qquad 
 \lim_{|h|\to \infty} 
 \sup_{\substack{x\in\T^2, t\in[0,T]}}w(x,h,t) =0.
 \label{eq: limits w}
\end{equation}
Then
\begin{equation}\label{eq:joint-maximum}
 \sup_{x\in\T^2,\,h\in\R^2}w(x,h,t)
 \le
 \sup_{x\in\T^2,\,h\in\R^2}w(x,h,0),
 \qquad 0\le t\le T.
\end{equation}
\end{lemma}

\begin{proof}
Let
\[
 M_0:=\sup_{x\in\T^2,\,h\in\R^2} w(x,h,0)  .
\]
Suppose that \eqref{eq:joint-maximum} fails at some time $T_0\in(0,T]$, that is 
\begin{equation*}
     \sup_{x\in\T^2,\,h\in\R^2} w(x,h, T_0)  
 > M_0 , 
\end{equation*}
and define $w_\ep := w- \ep t$. By \eqref{eq: limits w}, the maximum of
$w_\varepsilon$ on $\T^2\times\R^2\times[0,T_0]$ is attained at a point
$(x_*,h_*,t_*)$, and for $\ep$ sufficiently small its value is still larger than $M_0$; hence $h_*\ne0$ and
$t_*>0$. At this maximum point
\[
 \Lambda_x w_\varepsilon \ge 0, \qquad \nabla_x w_\varepsilon=0,
 \qquad
 \nabla_h w_\varepsilon=0,
 \qquad
 \partial_t w_\varepsilon\ge0,
\]
where at $t_*=T_0$ the last inequality is understood as a left
derivative.

Putting together the preceding inequalities, we obtain
\[
 \bigl(\partial_t+b\cdot\nabla_x+c\cdot\nabla_h
       +\Lambda_x\bigr) w_\varepsilon(x_*,h_*,t_*)\ge0.
\]
On the other hand, \eqref{eq:joint-subsolution} gives everywhere with
$h\ne0$ and $t>0$
\[
 \bigl(\partial_t+b\cdot\nabla_x+c\cdot\nabla_h
       +\Lambda_x\bigr) w_\varepsilon
 \le-\varepsilon<0,
\]
contradiction. 
\end{proof}

We shall also need the following standard Poisson kernel bounds.
\begin{lemma}
    Let $f \in C^\infty(\T^2)$. For all $z>0$ and $\alpha \in (0,1)$ we have 
\begin{align}
 \|\nabla P_z f\|_{L^\infty}+\|\nabla\cR P_z f\|_{L^\infty}
 &\le Cz^{-1}\|f\|_{L^\infty},\label{eq:poisson-Linf}\\
 \|\nabla P_z f\|_{L^\infty}+\|\nabla\cR P_z f\|_{L^\infty}
 &\le C_\alpha z^{\alpha-1}[f]_{C^\alpha}.\label{eq:poisson-holder}
\end{align}
\end{lemma}
\begin{proof}
The Euclidean Poisson kernel in dimension two is given by
\[
 p_z(y)=\frac{1}{2\pi}\frac{z}{(z^2+|y|^2)^{3/2}}, 
\]
and the periodic Poisson kernel on $\T^2$ is its periodization
\begin{equation}\label{eq: periodic Poisson}
 P_z(y)=\sum_{ k \in\mathbb Z^2}p_z(y+2\pi k ).
\end{equation}
Just scaling of the Euclidean kernel and periodization gives, for $m\ge 1$, 
\begin{align}
 \| \nabla^m P_z\|_{L^1(\mathbb T^2)}
 &\lesssim_m  z^{-m},                                      \label{eq:poisson-kernel-L1}\\
 \int_{\mathbb T^2}|y|^\alpha
       | \nabla^m P_z(y)|\,dy
 &\lesssim_{m, \alpha} z^{\alpha-m},
 \qquad 0<\alpha<1.                                    \label{eq:poisson-kernel-moment}
\end{align}
Indeed, using \eqref{eq: periodic Poisson} and that $|y| \le |y+2\pi k|$ for $y \in \T^2$ and $k\in \mathbb{Z}^2$, the left-hand sides are respectively bounded by 
\[
 \int_{\mathbb R^2}| \nabla^m p_z(y)|\,dy
 \quad\text{and}\quad
 \int_{\mathbb R^2}|y|^\alpha|\nabla^m p_z(y)|\,dy,
\]
and the asserted bounds follow by scaling with the change of variables \(y=z\xi\). Thus, \eqref{eq:poisson-kernel-L1} and
\eqref{eq:poisson-kernel-moment} imply
\begin{align}
 \|\nabla ^mP_zf\|_{L^\infty}
 &\lesssim_m z^{-m}\|f\|_{L^\infty},                       \label{eq:poisson-Dm-Linf}\\
 \|\nabla ^mP_zf\|_{L^\infty}
 &\lesssim_{m,\alpha}z^{\alpha-m}[f]_{C^\alpha}.
 \label{eq:poisson-Dm-holder}
\end{align}
Taking \(m=1\) gives the required bounds for \(\nabla P_zf\).

For the Riesz-transform terms, since \(R_j\Lambda=\partial_j\) and $R_j$ vanishes on constants, we have
\[
 \nabla R_jP_zf = - \nabla R_j \int_z^\infty \partial_s P_sf \, ds = \nabla R_j \int_z^\infty \Lambda P_s f \, ds 
 =\int_z^\infty \nabla \partial_j P_sf\,ds.
\]
Therefore, using \eqref{eq:poisson-Dm-Linf} with \(m=2\) gives
\[
 \|\nabla\cR P_zf\|_{L^\infty}
 \lesssim  \|f\|_{L^\infty}\int_z^\infty s^{-2}\,ds
 = z^{-1}\|f\|_{L^\infty},
\]
while \eqref{eq:poisson-Dm-holder} gives 
\[
 \|\nabla\cR P_zf\|_{L^\infty}
 \lesssim_\alpha [f]_{C^\alpha}
       \int_z^\infty s^{\alpha-2}\,ds
 = C_\alpha z^{\alpha-1}[f]_{C^\alpha} ,
\]
which completes the proof.
\end{proof}

\vspace{5pt}
\noindent\textbf{AI usage statement.}
The author used OpenAI's GPT-5.6 Sol in developing and proofreading this work. The author verified each step independently and rewrote it in the form presented here. The author assumes full responsibility for the correctness of the work. Access to the model was provided through OpenAI's ChatGPT for Academic Researchers program.

\bibliography{references}

@article {CZV16,
    AUTHOR = {Coti Zelati, Michele and Vicol, Vlad},
     TITLE = {On the global regularity for the supercritical {SQG} equation},
   JOURNAL = {Indiana Univ. Math. J.},
  FJOURNAL = {Indiana University Mathematics Journal},
    VOLUME = {65},
      YEAR = {2016},
    NUMBER = {2},
     PAGES = {535--552},
      ISSN = {0022-2518,1943-5258},
   MRCLASS = {35Q86 (35B65 35Q35 76B65 86A05 86A10)},
  MRNUMBER = {3498176},
MRREVIEWER = {Leonardo\ Marazzi},
       DOI = {10.1512/iumj.2016.65.5807},
       URL = {https://doi.org/10.1512/iumj.2016.65.5807},
}

@article{CG26,
  author        = {Caselli, M. and Gennaioli, L.},
  title         = {A new {D}uhamel-type principle with applications to geometric (in)equalities},
  journal       = {preprint arXiv:2603.29823},
  year          = {2026},
  eprint        = {2603.29823},
  archivePrefix = {arXiv},
  primaryClass  = {math.AP}
}

@article{CV12,
  author  = {Constantin, P. and Vicol, V.},
  title   = {Nonlinear maximum principles for dissipative linear nonlocal operators and applications},
  journal = {Geom. Funct. Anal.},
  volume  = {22},
  year    = {2012},
  pages   = {1289--1321}
}

@article{KNV07,
  author  = {Kiselev, A. and Nazarov, F. and Volberg, A.},
  title   = {Global well-posedness for the critical 2D dissipative quasi-geostrophic equation},
  journal = {Invent. Math.},
  volume  = {167},
  year    = {2007},
  pages   = {445--453}
}

@article{CMT94,
  author  = {Constantin, Peter and Majda, Andrew J. and Tabak, Esteban},
  title   = {Formation of strong fronts in the 2-D quasigeostrophic thermal active scalar},
  journal = {Nonlinearity},
  volume  = {7},
  number  = {6},
  pages   = {1495--1533},
  year    = {1994},
  doi     = {10.1088/0951-7715/7/6/001}
}

@article{CC04,
  author  = {C{\'o}rdoba, Antonio and C{\'o}rdoba, Diego},
  title   = {A Maximum Principle Applied to Quasi-Geostrophic Equations},
  journal = {Communications in Mathematical Physics},
  volume  = {249},
  pages   = {511--528},
  year    = {2004},
  doi     = {10.1007/s00220-004-1055-1}
}

@article{CV10,
  author  = {Caffarelli, Luis A. and Vasseur, Alexis},
  title   = {Drift diffusion equations with fractional diffusion and the quasi-geostrophic equation},
  journal = {Annals of Mathematics},
  volume  = {171},
  number  = {3},
  pages   = {1903--1930},
  year    = {2010},
  doi     = {10.4007/annals.2010.171.1903}
}

@article{CTV15,
  author  = {Constantin, Peter and Tarfulea, Andrei and Vicol, Vlad},
  title   = {Long Time Dynamics of Forced Critical {SQG}},
  journal = {Communications in Mathematical Physics},
  volume  = {335},
  number  = {1},
  pages   = {93--141},
  year    = {2015},
  doi     = {10.1007/s00220-014-2129-3}
}

@article{RoncalStinga16,
  author  = {Roncal, Luz and Stinga, Pablo Ra{\'u}l},
  title   = {Fractional {L}aplacian on the torus},
  journal = {Commun. Contemp. Math.},
  volume  = {18},
  number  = {3},
  year    = {2016},
  pages   = {1550033}
}

@article {CS07,
    AUTHOR = {Caffarelli, Luis and Silvestre, Luis},
     TITLE = {An extension problem related to the fractional {L}aplacian},
   JOURNAL = {Comm. Partial Differential Equations},
  FJOURNAL = {Communications in Partial Differential Equations},
    VOLUME = {32},
      YEAR = {2007},
    NUMBER = {7-9},
     PAGES = {1245--1260},
      ISSN = {0360-5302,1532-4133},
   MRCLASS = {35J70},
  MRNUMBER = {2354493},
MRREVIEWER = {Francesco\ Petitta},
       DOI = {10.1080/03605300600987306},
       URL = {https://doi.org/10.1080/03605300600987306},
}

@article{CordobaCordoba2003,
  author  = {C{\'o}rdoba, Antonio and C{\'o}rdoba, Diego},
  title   = {A pointwise estimate for fractionary derivatives with applications to partial differential equations},
  journal = {Proceedings of the National Academy of Sciences of the United States of America},
  volume  = {100},
  number  = {26},
  pages   = {15316--15317},
  year    = {2003},
  doi     = {10.1073/pnas.2036515100}
}

@article {KN09,
    AUTHOR = {Kiselev, A. and Nazarov, F.},
     TITLE = {A variation on a theme of {C}affarelli and {V}asseur},
   JOURNAL = {Zap. Nauchn. Sem. S.-Peterburg. Otdel. Mat. Inst. Steklov.
              (POMI)},
  FJOURNAL = {Rossi\u iskaya Akademiya Nauk. Sankt-Peterburgskoe Otdelenie.
              Matematicheski\u i\ Institut im. V. A. Steklova. Zapiski
              Nauchnykh Seminarov (POMI)},
    VOLUME = {370},
      YEAR = {2009},
     PAGES = {58--72, 220},
      ISSN = {0373-2703},
   MRCLASS = {35Q35 (35B65 35R11 76B03)},
  MRNUMBER = {2749211},
MRREVIEWER = {Benedetta\ Ferrario},
       DOI = {10.1007/s10958-010-9842-z},
       URL = {https://doi.org/10.1007/s10958-010-9842-z},
}
	 	 \bibliographystyle{alpha}

\end{document}